\documentclass[12pt]{amsart}
\usepackage{amsmath}
\usepackage{amssymb}
\usepackage{amscd}
\usepackage[mathscr]{eucal}
\usepackage{latexsym}
\usepackage{color}

\newtheorem{thm}{Theorem}[section]
\newtheorem{lem}[thm]{Lemma}

\theoremstyle{definition} 

\newtheorem{rem}[thm]{Remark}

\newtheorem{ex}[thm]{Example}

\theoremstyle{remark}

\numberwithin{equation}{section}

\def\Sha{\text{\rm Sha}}

\def\Art{\text{\rm Art}}
\def\Sha{\text{\rm Sha}}

\def\mod{\text{\rm mod}}

\def\-rig{\text{\rm -rig}}
\def\-log{\text{\rm -log}}
\def\-dif{\text{\rm -dif}}

\def\rank{\text{\rm rank}}
\def\ord{\text{\rm ord}}
\def\Aut{\text{\rm Aut}}
\def\Gal{\text{\rm Gal}}
\def\Hom{\text{\rm Hom}}

\def\Ker{\text{\rm Ker}\,}

\def\ab{\text{\rm ab}}
\def\Frob{\text{\rm Frob}}

\def\End{\text{\rm End}}
\def\Sel{\text{\rm Sel}}

\def\lim{\text{\rm lim}}

\def\rnum#1{\expandafter{\romannumeral #1}} 
\def\Rnum#1{\uppercase\expandafter{\romannumeral #1}} 

\makeatletter
\def\widebar{\accentset{{\cc@style\underline{\mskip10mu}}}}
\def\Widebar{\accentset{{\cc@style\underline{\mskip8mu}}}}
\makeatother

\begin{document}

\title[ BSD conjecture in the CM case and the congruent number problem]
{Birch and Swinnerton-Dyer conjecture in the complex multiplication case and the congruent number problem}

\author[Kazuma Morita]{Kazuma Morita}

\subjclass{ %2000 MSC number
11F03, 11G05, 11G40}
\keywords{ %key words and phrases
classical number theory, elliptic curves}
\date{\today}

\maketitle

{\bf Abstract.}  For an elliptic curve $E$ over $K$, the Birch and Swinnerton-Dyer conjecture predicts that the rank of Mordell-Weil group $E(K)$ is equal to the order of the zero of $L(E_{\slash K},s)$ at $s=1$.
In this paper, we shall give a proof for elliptic curves with complex multiplications. The key method of the proof  is to reduce the Galois action of infinite order on the Tate module of an elliptic curve to that of finite order by using the $p$-adic Hodge theory. As a corollary, we can determine whether a given natural number is a congruent number (congruent number problem). This problem is one of the oldest unsolved problems in mathematics.

\section{Introduction}

A natural number $N$ is called a congruent number if it is the area of a rational right-angled triangle. To determine whether $N$ is a congruent number is a difficult problem which can be traced back at least a millennium and is the oldest major unsolved problem in number theory. For example, it is easy to verify that $6$ is a congruent number since the sides lengths $(3,4,5)$ give the right-angled triangle whose area is $6$. Furthermore, $5$ and $7$ are shown to be congruent numbers by Fibonacci and Euler. On the other hand, Fermat showed that the square numbers are never congruent numbers and,  based on this experience, led to the famous Fermat's conjecture. In modern language, a natural number $N$ is a congruent number if and only if there exists a rational point $(x,y)$ with $y\not=0$ on the elliptic curve $y^2=x^3-N^2x$. The Birch and Swinnerton-Dyer conjecture which is one of the millennium problems by Clay Mathematics Institute concerns the rational points on the elliptic curves. The substantial developments are given by Coates-Wiles, Gross-Zagier, Kolyvagin, Rubin, and so on.  In this paper, we shall give a proof for elliptic curves with complex multiplications which 
include the elliptic curves of the form  $y^2=x^3-N^2x$ and, as a corollary, can determine whether $N$ is a congruent number by the criterion of Tunnell.

{\bf Organization.} Let $E/K$ be an elliptic curve with complex multiplication over an imaginary quadratic field $K$. Choose an algebraic closure $\overline{K}$ of $K$ and consider the absolute Galois group $G_{K}=\Gal(\overline{K}/K)$. For a prime $p$, consider the Tate module $\mathbb{V}_p(E)$ over $\mathbb{Q}_{p}$ and put 
$V_{p}(E)=\mathbb{V}_{p}(E)$ if $p$ splits in $K$ (resp. $V_{p}(E)=\mathbb{V}_{p}(E)\otimes_{\mathbb{Q}_p}K_p$ otherwise).
Then, by the theory of Lubin-Tate, we have a splitting $ V_{p}(E)=V^{(1)}\oplus V^{(2)}$ of $G_{K}$-modules where $V^{(i)}$ is a $1$-dimensional vector space. In Section 3.2, 
we show that $V^{(i)}$ has the Hodge-Tate weight $(0,0)$ or $(1,1)$ in the sense of Section 2. Note that, contrary to the Hodge weights over $\mathbb{C}$, the action of $\Gal(K/\mathbb{Q})$ (and its restriction $\Gal(K_p/\mathbb{Q}_p)$) does not exchange the Hodge-Tate weights (see Remark 3.2 and 3.4). The key method of the proof of the main theorem is to reduce the Galois action of infinite order on the Tate module $V_{p}(E)$ to that of finite order by using the fact that $V^{(i)}$ has the Hodge-Tate weight $(0,0)$ or $(1,1)$ (see Section 2 for details).  This means that the proof of the main theorem is reduced to the methods in the algebraic number theory.

\section{Algebraic Hecke character and Algebraic Galois character}
In this section, we shall review the basic facts on the class field theory based on [Y]. Let $F$ be a number field and $\mathbb{A}_{F}$ denote the 
adele ring of $F$. For each place $v$ of $F$, consider its completion $F_{v}$. If $v$ is an infinite place of $F$, denote the connected component 
of identity of $F_{v}^{\times}$ by 
$F_{v}^{\times 0}$ and define $F_{\infty}^{\times 0}=\prod_{v\mid \infty}F_{v}^{\times 0}$. Then, the global class field 
theory claims that the global Artin map $\Art_{F}$ induces the isomorphism 
$$\Art_{F}:\mathbb{A}_{F}^{\times}/\overline{F^{\times}F_{\infty}^{\times 0}}\simeq G_{F}^{\ab}$$
where $-$ denotes the closure in $\mathbb{A}_{F}^{\times}$. Since the kernel of a continuous character of $\mathbb{A}_{F}^{\times}/F^{\times}$ is a closed set in $\mathbb{A}_{F}^{\times}/F^{\times}$, we have  a one-to-one correspondence
$$\{\text{Galois character}\ R:G_{F}\rightarrow \mathbb{C}^{\times}\}\leftrightarrow \{\text{Hecke character}\ \Pi:\mathbb{A}_{F}^{\times}/F^{\times}\rightarrow \mathbb{C}^{\times}, \Pi_{\mid F_{\infty}^{\times 0}}=1\}$$
where this correspondence is given by $\Pi=R\circ \Art_{F}$. Now, let us consider a $p$-adic Hecke character $\Pi:\mathbb{A}_{F}^{\times}/F^{\times}\rightarrow \overline{\mathbb{Q}_{p}}^{\times}$. Since $\overline{\mathbb{Q}_{p}}^{\times}$ is totally disconnected, the connected component of the identity on the left hand side is contained in $\Ker (\Pi)$ and we obtain $\Pi_{\mid F_{\infty}^{\times 0}}=1$. By the same method above, we have a one-to-one correspondence
$$\{\text{Galois character}\ R:G_{F}\rightarrow \overline{\mathbb{Q}_{p}}^{\times}\}\leftrightarrow \{\text{Hecke character}\ \Pi:\mathbb{A}_{F}^{\times}/F^{\times}\rightarrow \overline{\mathbb{Q}_{p}}^{\times} \}.$$
\subsection{Algebraic Hecke character} 
Let $\Pi:\mathbb{A}_{F}^{\times}/F^{\times}\rightarrow \mathbb{C}^{\times}$ be a Hecke character and denote its restriction to $F_{v}^{\times}$ by $\Pi_{v}$. For an infinite place $v$ of $F$, we say that $\Pi_{v}$ is  algebraic if there exist integers $a_{\tau}\in\mathbb{Z}$ such that we have
$$\Pi_{v}(x_{v})=\prod_{\tau\in\Hom_{\mathbb{R}}(F_{v},\mathbb{C})}\tau(x_{v})^{a_{\tau}}\quad (\forall x_{v}\in F_{v}^{\times 0}).$$
Furthermore, if $\Pi_{v}$ is algebraic for all infinite places $v$ in $F$, we say that $\Pi$ is the {\bf algebraic Hecke character} (different from the usual definition) and that the set of 
integers $(a_{\tau})_{\tau}$ (the number of this set is $[F:\mathbb{Q}]$) is the weights of $\Pi$. For an algebraic Hecke character $\Pi$ of weights 
$(a_{\tau})_{\tau}$ and a field isomorphism $\iota:\overline{\mathbb{Q}}_{p}\simeq \mathbb{C}$, define the $p$-adic Hecke character $\Pi_{\iota}$ by 
$$\Pi_{\iota}: x \ (\in \mathbb{A}^{\times}_{F})\mapsto \iota^{-1}\Bigl(\Pi(x)\cdot  \prod_{\tau\in\Hom_{\mathbb{Q}}(F,\mathbb{C})}\tau(x_{v(\tau)})^{-a_{\tau}}\Bigr)\cdot \prod_{\tau\in\Hom_{\mathbb{Q}}(F,\overline{\mathbb{Q}_{p}})}\tau(x_{v(\tau)})^{a_{\iota\circ \tau}} (\in \overline{\mathbb{Q}_{p}}^{\times}).$$
The $p$-adic Hecke character $\Pi_{\iota}$ factors through $\mathbb{A}^{\times}_{F}/F^{\times}$ and we have $(\Pi_{\iota})_{v}=\iota^{-1}\circ\Pi_{v}$ for $v\nmid p,\infty$. For a finite place $v$ of $F$, since $\Pi_{v}$ factors through a discrete group, $(\Pi_{\iota})_{v}$ is continuous. On the other hand,  for an infinite place $v$, since $(\Pi_{\iota})_{v}$ factors through a finite group, $(\Pi_{\iota})_{v}$ is also continuous. It follows that $\Pi_{\iota}$ is a continuous character.  
\subsection{Algebraic Galois character} 
For an algebraic Hecke character $\Pi$ of weights 
$(a_{\tau})_{\tau}$ and a field  isomorphism $\iota:\overline{\mathbb{Q}}_{p}\simeq \mathbb{C}$, define the $p$-adic Galois character $R$ by 
$\Pi_{\iota}=R\circ \Art_{F}$. Denote its restriction to the Weil group of $F_{v}$ by $R_{v}$ and put $\Art_{v}=\Art_{F}\mid_{F_v^{\times}}$. It follows that we have $(\Pi_{\iota})_v=R_v\circ \Art_{v}$. Then, for 
$v\mid p$ and $\tau\in \Hom_{\mathbb{Q}_{p}}(F_{v},\overline{\mathbb{Q}_{p}})$, there exist $b_{\tau}:=a_{\iota\circ \tau}\in \mathbb{Z}$ and an open 
subgroup $U\subset \mathscr{O}^{\times}_{v}$ ($\mathscr{O}_{v}$: the ring of integers of $F_v$) such that we have
$$R_{v}\circ \Art_{v}(x_{v})=\prod_{\tau\in\Hom_{\mathbb{Q}_{p}}(F_{v},\overline{\mathbb{Q}_{p}})}\tau(x_{v})^{b_{\tau}}\quad (\forall x_{v}\in U).$$
Note that, since the finite component of $\Pi:\mathbb{A}_{F}^{\times}/F^{\times}\rightarrow \mathbb{C}^{\times}$ is a character of a locally profinite group, the finite component of $\Pi$ is trivial on an open subgroup $\mathbb{U}$ 
of $\hat{\mathscr{O}}_{F}^{\times}:=(\mathscr{O}_{F}\otimes_{\mathbb{Z}}\hat{\mathbb{Z}})^{\times}$ ($\mathscr{O}_{F}$: the ring of integers of $F$).
We say that such a $p$-adic Galois character $R$ is an {\bf algebraic Galois character} and that the set of 
integers $(b_{\tau})_{\tau}$ (the number of this set is $[F:\mathbb{Q}]$) is the Hodge-Tate weights of $R$. 
Conversely, we can obtain the algebraic Hecke character $\Pi$ from the algebraic Galois character $R$ by reversing the procedure. 
In particular, it follows from the global class field theory that the  algebraic Galois character of Hodge-Tate weights $(0)_{\tau}$ factors through a finite quotient since the corresponding algebraic Hecke character $\Pi$ satisfies $\Pi_{\mid F_{\infty}^{\times 0}}=1$.
\section{Elliptic curve and class field theory}
\subsection{Preliminary} 
Let $E/K$ be an elliptic curve with complex multiplication over an imaginary quadratic field $K$. Choose an algebraic closure $\overline{K}$ of $K$ and consider the absolute Galois group $G_{K}=\Gal(\overline{K}/K)$. For the Tate module $\mathbb{V}_p(E)$ over $\mathbb{Q}_{p}$,  put 
$V_{p}(E)=\mathbb{V}_{p}(E)$ if $p$ splits in $K$ (resp. $V_{p}(E)=\mathbb{V}_{p}(E)\otimes_{\mathbb{Q}_p}K_p$ otherwise). 
Then, by the theory of Lubin-Tate, we have a splitting $ V_{p}(E)=V^{(1)}\oplus V^{(2)}$ of $G_{K}$-modules where $V^{(i)}$ is a $1$-dimensional vector space. 
\begin{lem}
With notations as above,  $V^{(i)}$ has the Hodge-Tate weight $0$ or $1$. 
\end{lem}
\begin{proof} Since the weight of an algebraic Hecke character (hence Hodge-Tate weight) does not depend on the choice of prime $p$, we may assume that the prime $p$ inerts in $K$. In this proof, we consider the tensor products over $K_p$.  Let $\mathbb{C}_p$ denote the $p$-adic completion of 
an algebraic closure of $\mathbb{Q}_{p}$. Since, by the comparison theorem of the $p$-adic Hodge theory ([T], p.180, Cor.2), we have $\mathbb{C}_p\otimes V_{p}(E)\simeq \mathbb{C}_p\oplus \mathbb{C}_p(-1)$, it suffices to show that we have  
$\mathbb{C}_p\otimes V^{(1)}\simeq \mathbb{C}_p$, $\mathbb{C}_p\otimes V^{(2)}\simeq \mathbb{C}_p(-1)$ or $\mathbb{C}_p\otimes V^{(1)}\simeq \mathbb{C}_p(-1)$, $\mathbb{C}_p\otimes V^{(2)}\simeq \mathbb{C}_p$. 
Let $\alpha$ (resp. $\beta$) denote an element of $V^{(1)}$ (resp. $V^{(2)}$). 
Put
$$(*)\qquad g(\alpha)=s\alpha\quad \text{ and}\quad g(\beta)=t\beta\quad (g\in G_{K_{p}}, \ s,t\in K_{p}^{\times}).$$ 
On the other hand, since $1\otimes \alpha$ and $1\otimes \beta$ are elements of $\mathbb{C}_{p}\otimes V_{p}\simeq \mathbb{C}_{p}\oplus \mathbb{C}_{p}(-1)$, we can write
$$(**)\qquad1\otimes \alpha=a {\bf 1}+b {\bf T} \quad \text{and}\quad1\otimes \beta=c {\bf 1}+d {\bf T}\quad (a,b,c,d\in \mathbb{C}_{p})$$ 
where ${\bf 1}$ (resp. ${\bf T}$) denotes a basis of $\mathbb{C}_{p}$ (resp. $\mathbb{C}_{p}(-1)$) such that we have  $g({\bf 1})={\bf 1}$ and $g({\bf T})=\chi(g){\bf T}$ ($g\in G_{K_{p}}$, $\chi$ is the cyclotomic character). From the presentation of $(*)$ and $(**)$, we have
$$g(a) {\bf 1}+g(b) \chi(g){\bf T}=s(a {\bf 1}+b {\bf T})\ \text{and}\ g(c) {\bf 1}+g(d) \chi(g){\bf T}=t(c {\bf 1}+d{\bf T})$$ 
and it follows that we obtain 
$$g(a)=sa,\ g(b)\chi(g)=sb\quad \text{and} \quad g(c)=tc,\ g(d)\chi(g)=td.$$
Assume that we have $b\not=0$ and then we have $g(ab^{-1})=\chi(g)ab^{-1}$. Since such an element exists in $\mathbb{C}_p$ if and only if  $ab^{-1}=0$, this means that we have $a=0$. 
Therefore, we can obtain $a=0$ or $b=0$ (similarly $c=0$ or $d=0$) and it follows from $(**)$ that $V^{(i)}$ has the Hodge-Tate weight $0$ or $1$.
\end{proof}
\begin{rem} With notations as in the proof of Lemma 3.1, we show that each Hodge-Tate weight does not depend on the embeddings $K_p\hookrightarrow \mathbb{C}_p$.  Let $h:K_p\hookrightarrow \mathbb{C}_p$ be the other embedding and $\mathbb{C}_{p}\otimes_{K_p}V_{p}(E)$ is equipped with the action of $\Gal(K_p/\mathbb{Q}_p)$ by $\mathbb{C}_{p}\otimes_{h(K_p)}V_{p}(E)$. 
Put $H=\Gal(\overline{K_p}/\mathbb{Q}_{p}^{\ab})$ and $\Gamma=G_{K_{p}}/H$. Then, $D:=(\mathbb{C}_p\otimes V_p(E))^{H}=\mathbb{C}_p^H\oplus \mathbb{C}_p^{H}(-1)$ is equipped with the action of $\Gamma$. Since $h\in \Gal(K_p/\mathbb{Q}_p)=\Gal(\mathbb{Q}_p^{\ab}/\mathbb{Q}_p)/\Gamma$ commutes with the action of $\Gamma$, this induces the $\Gamma$-equivariant map 
$$h: D\rightarrow h(D)\  (\subset \mathbb{C}_p\otimes V_p(E)).$$
Then, we can write 
$$(***)\qquad h({\bf 1})=k{\bf 1}+l{\bf T}\quad\text{and} \quad h({\bf T})=m{\bf 1}+n{\bf T}\quad (k,l,m,n\in\mathbb{C}_{p}).$$
On the other hand, for an element $g\in \Gamma=\Gal(\mathbb{Q}_{p}^{\ab}/K_{p})$, we have   

1). $gh({\bf 1})=g(k){\bf 1}+g(l)\chi(g){\bf T}=k{\bf 1}+l{\bf T}=hg({\bf 1})$, 

2). $gh({\bf T})=g(m){\bf 1}+g(n)\chi(g){\bf T}=\chi(g)m{\bf 1}+\chi(g)n{\bf T}=hg({\bf T})$.    

Thus, we obtain $g(l)=\chi(g)^{-1}l$ and $g(m)=\chi(g)m$ and such elements $l,m$ exist in $\mathbb{C}_p$ if and only if  $l=m=0$. It follows from $(***)$ that we have $h(\mathbb{C}_p)=\mathbb{C}_p$ and $h(\mathbb{C}_p(-1))=\mathbb{C}_p(-1)$ as the action of $h$ on $\mathbb{C}_p\otimes V_{p}(E)= \mathbb{C}_p\oplus \mathbb{C}_p(-1)$.  
Therefore, each Hodge-Tate weight does  not depend on the embeddings $K_p\hookrightarrow \mathbb{C}_p$. This means that the Hodge-Tate weights of   $V^{(i)}$ as a $\mathbb{Q}_p$-representation are $(0,0)$ or $(1,1)$.     
\end{rem}
From now on, fix the notation such that  $V^{(1)}$ (resp. $V^{(2)}$) has the Hodge-Tate weight $0$ (resp. $1$).  
The following is the key lemma which connects the unit elements of the rings of integers of number fields and rational points on elliptic curves.   
\begin{lem}
In the sense of $\S$2, $V^{(1)}$ (resp. $V^{(2)}$) has the Hodge-Tate weight $(0)_{\tau}$ (resp. $(1)_{\tau}$).
\end{lem}
\begin{proof} As in the proof of Lemma 3.1, we may assume that the prime $p$ inerts in $K$ and we consider the Tate module $V_{p}(E)$ over $K_p$. First of all, note that we have the one-to-one correspondence between the algebraic Hecke character (the absolute value) $\mid\cdot\mid:\mathbb{A}_{\mathbb{Q}}^{\times}/\mathbb{Q}^{\times}\rightarrow \mathbb{R}^{\times}_{>0}\subset\mathbb{C}^{\times}$ of weight $1$ and the algebraic Galois character (the cyclotomic character) $\chi:G_{\mathbb{Q}}\rightarrow  \overline{\mathbb{Q}_{p}}^{\times}$ of Hodge-Tate weight $1$ in the sense of $\S$2. Further, by the base change for a number field $F/\mathbb{Q}$, we also  have the one-to-one correspondence between $\mid\cdot\mid\circ N_{F/\mathbb{Q}}$ and $\chi_{\mid F}$ where $N_{F/\mathbb{Q}}$ denotes the norm map from $\mathbb{A}_{F}^{\times}$ to $\mathbb{A}_{\mathbb{Q}}^{\times}$. Let $\sigma^{(i)}:G_{K}\rightarrow \Aut(V^{(i)})$ denote the Galois character obtained from the action of $G_{K}$ on $V_{p}(E)$ and let $\sigma^{(i)}_{p}$ be its restriction to the Weil group of $K_{p}$. Denote the ring of integers of $K_p$ by $\mathscr{O}_{p}$. Since $V^{(1)}$ (resp. $V^{(2)}$) has the Hodge-Tate weight $0$ (resp. $1$), there exists an open subgroup $U$ of  $\mathscr{O}_{p}^{\times}$ such that we have ([F], p.143, 3.9.iv) 
$$\sigma^{(1)}_{p}\circ \Art_{p}(x)=N_{K_p/\mathbb{Q}_p}(1)=1 \ \text{and} \ \sigma^{(2)}_{p}\circ \Art_{p}(x)=N_{K_p/\mathbb{Q}_p}(x)=x\tau(x)\quad (\forall x\in U)$$
where $N_{K_p/\mathbb{Q}_p}$ denotes the norm map from $K_p$ to $\mathbb{Q}_p$ and $1\not=\tau\in\Hom_{\mathbb{Q}_{p}}(K_p,\overline{\mathbb{Q}_{p}})$ (see $\S$2). This means that $V^{(1)}$ (resp. $V^{(2)}$) has the Hodge-Tate weight $(0)_{\tau}$ (resp. $(1)_{\tau}$).
\end{proof}
\subsection{Global class field theory and Hecke character}
Let $\sigma^{(i)}:G_{K}\rightarrow \Aut(V^{(i)})$ ($i=1,2$) denote the Galois character obtained from the action of $G_{K}$ on $V_{p}(E)$. 
By the theory of complex multiplication ([S], \Rnum{2}.10.5), there exists a Hecke character $\omega_{K}$ over $K$ such that, for a good prime $v\nmid p$, the prime element $\wp_v$ of $K$ and the Frobenius element $\Frob_v$ are related by
$$(\sharp)\quad \omega_{K,v}(\wp_v)=\iota\circ \sigma^{(1)}_v(\Frob_v)\quad  \text{and}\quad \overline{\omega_{K,v}}(\wp_v)=\iota\circ \sigma^{(2)}_v(\Frob_v)$$
for a field isomorphism $\iota:\overline{\mathbb{Q}}_{p}\simeq \mathbb{C}$. 
\begin{rem}
Since the both sides are connected just by a field isomorphism $\iota:\overline{\mathbb{Q}}_{p}\simeq \mathbb{C}$, this does not mean that the effect of the complex conjugation changes  $\sigma^{(1)}$ to $\sigma^{(2)}$ as a continuous character and we shall normalize the both sides of $(\sharp)$ to obtain the continuous Hecke and Galois characters (see $\S$2). 
\end{rem} 
Let us denote $\omega^{(1)}_{K}=\omega_{K}$ and $\omega^{(2)}_{K}=\overline{\omega_{K}}$. Since $V^{(1)}$ (resp. $V^{(2)}(1)$) has the Hodge-Tate weight $(0)_{\tau}$, by making the infinite component of $\omega^{(1)}_{K}$ (resp. $\omega^{(2)}_{K}(1)$) trivial, we can obtain the {\bf algebraic Hecke character} $\omega^{(1)}_{K,a}$ (resp. $\omega^{(2)}_{K,a}(1)$) of the weight $(0)_{\iota\circ\tau}$. Conversely, the Galois character which corresponds to the algebraic Hecke character $\omega^{(1)}_{K,a}$ (resp. $\omega^{(2)}_{K,a}(1)$) becomes the {\bf algebraic Galois character} $\sigma^{(1)}_{a}$ (resp. $\sigma^{(2)}_{a}(1)$) of the Hodge-Tate weight $(0)_{\tau}$ and factors through a finite abelian extension $K_{1}/K$ (resp. $K_{2}/K$) by the global class field theory. In this situation, the effect of the complex conjugation changes $\sigma^{(1)}_{a}$ (resp.  $\omega^{(1)}_{K,a}$) to  $\sigma^{(2)}_{a}(1)$ (resp. $\omega^{(2)}_{K,a}(1)$) as a continuous character.
\begin{rem}
Note that these algebraic characters are different from the former ones up to the normalization and that the former characters never factor through finite quotients since the Galois actions on the Tate module $V_{p}(E)$ are of infinite order. 
\end{rem}
\begin{ex} Consider the elliptic curve $y^2=x^3-Dx$ ($D\in\mathbb{Z}\backslash \{ 0\}$) which has the complex multiplication by $K=\mathbb{Q}(i)$. Let $p$ denote the prime ideal of $\mathbb{Z}$ such that $p$ does not divide $2D$.  By making the infinite components of the Hecke characters trivial, we shall obtain the algebraic Hecke characters. 

a). In the case of $p\equiv 3\ (\mod \ 4)$, 
we have $\omega^{(1)}_{K}(p)=-p$ and we normalize this character as $\omega^{(1)}_{K,a}(p)=-1$. Then, its normalized complex conjugation  $\omega^{(2)}_{K,a}(1)(p)$ is also given by $\omega^{(2)}_{K,a}(1)(p)=-1$. 

b). Consider the case of $p\equiv 1\ (\mod \ 4)$ and then the prime $p$ decomposes as 
$p=\pi\bar{\pi}$ where $\pi$ denotes the element of $\mathbb{Z}[i]$ such that we have  $\pi\equiv 1\ (\mod \ 2+2i).$ Let $\wp$ denote the prime ideal $\wp=(\pi)$ in $\mathbb{Z}[i]$. It is known that we have
$\omega^{(1)}_{K}(\wp)=\overline{\bigl(\frac{D}{\pi}\bigr)}_{4}\pi$ or $\bigl(\frac{D}{\pi}\bigr)_{4}\bar{\pi}$ (depending on the choice of $\iota:\overline{\mathbb{Q}}_{p}\simeq \mathbb{C}$) where $\bigl(\frac{\cdot}{\pi}\bigr)_{4}$ denotes the $4^{\text{th}}$-power residue symbol. 
Assume that we have $\omega^{(1)}_{K}(\wp)=\overline{\bigl(\frac{D}{\pi}\bigr)}_{4}\pi$. 
In this case, we normalize this character as $\omega^{(1)}_{K,a}(\wp)=\overline{\bigl(\frac{D}{\pi}\bigr)}_{4}$ and its normalized complex conjugation $\omega^{(2)}_{K,a}(1)(\wp)$ is given by  $\omega^{(2)}_{K,a}(1)(\wp)=\bigl(\frac{D}{\pi}\bigr)_{4}$. 

Since one of  these local eigenvalues of normalized Hecke characters is a root of unity (i.e. weight $0$), this forces the normalized Hecke $L$-function to be of weight $0$ and thus the normalized Hecke characters factor through global finite Galois extensions.  
\end{ex}
\subsection{Hecke $L$-function and Dedekind zeta function}   
Since the Galois character $\sigma^{(1)}_{a}$  is non-trivial, we can show that $L(\omega^{(1)}_{K,a},s)$ does not have any poles or zeros at $s=1$. Furthermore, we can deduce that the order of $L(  \omega^{(2)}_{K,a}(1),s)\mid_{s=0}$ is the same as that of Dedekind zeta function $\zeta_{K_{2}}(s)\mid_{s=0}$ ([N], p.502, (8.5)).
\section{Rational points and unit elements}
Choose a prime $p$ which inerts in $K$ such that the extensions $K_{1}/K$ and $K_{2}/K$ are of degree prime to $p$. Since $E$ has the complex multiplication by $K$ and the prime $p$ inerts in $K$, the Tate module $\mathbb{V}_p(E)$  over $\mathbb{Q}_{p}$ is a one dimensional $K_p$-vector space. Further, recall that  $V_p(E)$ over $K_p$ splits as $ V_{p}(E)=V^{(1)}\oplus V^{(2)}$. The assumption that the extensions $K_{1}/K$ and $K_{2}/K$ are of degree prime to $p$ is used in ([Ru], p.24-27).
\subsection{Top exact sequence}
Let $T^{(1)}$ denote a free of  rank $1$ module over $\mathscr{O}_{p}$ (the ring of integers of $K_{p}$) on which $G_{K}$ acts via $\sigma^{(1)}_{a}$  and put $W^{(1)}=(T^{(1)}\otimes \mathbb{Q})/T^{(1)}$. Then, it is known that we have $$\Sel(K,W^{(1)})\simeq \Hom(A_{K_{1}}, W^{(1)})^{\Gal(K_{1}/K)}$$
where $A_{L}$ denotes the ideal class group of $L$ ([Ru], p.24). Thus, we have an exact sequence  $0 \rightarrow \Ker (f) \rightarrow \Sel (K, W^{(1)}) \xrightarrow{f} \Sha(K, W^{(1)})_{p^{\infty}} \rightarrow 0$ of finite torsion modules over $K_p/\mathscr{O}_{p}$. 
\subsection{Middle exact sequence}
In $\S$3.2, for the Galois representation $\sigma_p:=\sigma^{(1)}\oplus \sigma^{(2)}$ of $G_{K}$ on $V_{p}(E)$ over $K_p$, we obtain the normalized Galois representation $\sigma_{p,a}:=\sigma^{(1)}_{a}\oplus \sigma^{(2)}_{a}$. Write the same notations for their restrictions to $\mathbb{V}_p(E)$. Since the Tate module $\mathbb{V}_p(E)$ over $\mathbb{Q}_{p}$ is a one dimensional $K_p$-vector space, we can write $\sigma_{p,a}=\sigma_{p}\cdot \psi_{p}$ for some character $\psi_{p}$ of $G_{K}$ over $K_{p}$. On the other hand, since we have $\End(E)\otimes \mathbb{Q}_{p}\simeq \End(\mathbb{V}_{p}(E)):\sigma\otimes 1\mapsto \sigma_{p}$ by the isogeny theorem of Faltings, there exists an element $\psi\in \End(E)$ such that we have $\End(E(\psi))\otimes \mathbb{Q}_{p}\simeq \End(\mathbb{V}_{p}(E)(\psi_p)):(\sigma\cdot\psi)\otimes 1\mapsto \sigma_{p}\cdot\psi_{p}=\sigma_{p,a}$.
\subsection{Bottom exact sequence}
Let $T^{(2)}$ denote a free of  rank $1$ module over $\mathscr{O}_{p}$ on which $G_{K}$ acts via $\sigma^{(2)}_a$  and put $W^{(2)}=(T^{(2)}\otimes \mathbb{Q})/T^{(2)}$. The following notation is used in ([Ru], p25). 

$\underline{\text{Notation}}$: For a $\mathbb{Z}[G_K]$-module $B$, we define the $\chi$-component of $B$ by 
$$B^{\chi}=\{b\in (B\otimes_{\mathbb{Z}}\mathbb{Z}_p)\otimes_{\mathbb{Z}_p}\mathscr{O}_p\mid 
g(b)=\chi(g)b\ (\forall g\in G_{K})\}.$$
With this notation, we have 
an exact sequence of  $K_p/\mathscr{O}_{p}$-modules ([Ru], p.26)
$$0 \rightarrow  (\mathscr{O}^{\times}_{K_2}\otimes \mathbb{Q}_{p}/\mathbb{Z}_{p})^{(\sigma^{(2)}_a(1))^{-1}}  \rightarrow \Sel (K, W^{(2)}) \rightarrow (A_{K_{2}})^{( \sigma^{(2)}_a (1))^{-1}}\rightarrow 0.$$
\subsection{Commutative diagram}
\begin{lem} We have $H^{0}(G_K,W^{(2)})=0$ and $H^{2}(G_K,W^{(1)})=0$.
\end{lem}
\begin{proof}Since $G_{K}$ acts on $W^{(2)}$ via  the character $\sigma^{(2)}_a$, the $G_K$-invariant part of $W^{(2)}$ is $0$ and thus  we obtain $H^{0}(G_K,W^{(2)})=0$. As for $H^{2}(G_K,W^{(1)})$, since $G_{K}$ acts on $W^{(1)}$ via the finite character $\sigma^{(1)}_a$, we may assume that the extension $K_1/K$ is the finite cyclic extension and we obtain $H^{2}(G_K,W^{(1)})=H^{2}(\Gal(K_1/K),W^{(1)})=(W^{(1)})^{\Gal(K_1/K)}/N_{\Gal(K_1/K)}(W^{(1)})=0$ where $N_G$ is the norm map $N_G(a)=\sum_{g\in G} ga$.
\end{proof}
With this lemma, we have the vanishing of the cohomology groups $\flat$ and $\natural$ below and thus we obtain 
the following commutative diagram of $K_p/\mathscr{O}_{p}$-modules where all of the horizontal and vertical lines are exact.   
{\tiny \begin{align*}
\begin{CD}
@. @. @. \flat=0 \\
@. @. @. @VVV \\
0 @>>> \Ker (f) @>>> \Sel (K, W^{(1)}) @> f>> \Sha(K, W^{(1)})_{p^{\infty}} @>>> 0 \\
@. @VVV @VVV @VVV \\
0 @>>> E(\psi)^{G_{K}}\otimes K_p/\mathscr{O}_{p} @>>> \Sel (K, E(\psi)_{p^{\infty}}) @>>> \Sha(K,E(\psi))_{p^{\infty}} @>>> 0 \\
@. @VVV @VVV @VVV \\
0 @>>> (\mathscr{O}^{\times}_{K_2}\otimes \mathbb{Q}_{p}/\mathbb{Z}_{p})^{(\sigma^{(2)}_a(1))^{-1}} @>>> \Sel (K, W^{(2)}) @>>> (A_{K_{2}})^{( \sigma^{(2)}_a(1))^{-1}} @>>> 0. \\
@. @. @VVV @.\\
@. @. \natural=0 @.\\
\end{CD}
\end{align*}}
\begin{lem}
If  we have $\ord_{s=1} L(E(\psi)_{\slash K},s)=\rank (E(\psi)^{G_{K}})$, this leads to  $\ord_{s=1} L(E_{\slash K},s)=\rank(E(K))$.
\end{lem} 
\begin{proof}
This follows from ([Ro], p.127) by reversing the procedure and using the fact that all conjugates 
$E(\psi)(\psi^{-1}_c)\simeq c^{-1}(E(\psi)(\psi^{-1}_c))= c^{-1}(E(\psi))(\psi^{-1})\simeq E(\psi)(\psi^{-1})=E \ (c\in\Aut(\mathbb{C}/\mathbb{Q}))$ give the same $L$-functions.
\end{proof}
\begin{lem} The rank of $\mathscr{O}^{\times}_{K_2}$ modulo torsions over $\mathbb{Z}$ is equal to that of $(\mathscr{O}^{\times}_{K_2}\otimes \mathbb{Q}_{p}/\mathbb{Z}_{p})^{(\sigma^{(2)}_a(1))^{-1}}$ over $K_p/\mathscr{O}_p$ (see the notation in $\S4.3$). 
\end{lem} 
\begin{proof}
For simplicity, write $\chi=\sigma^{(2)}_a(1)$ and $B=(\mathscr{O}^{\times}_{K_2}\otimes \mathbb{Q}_{p}/\mathbb{Z}_{p})\otimes_{\mathbb{Z}_p}\mathscr{O}_p$. Since $K_2$ is the fixed field of the kernel of the finite character $\chi$, we may assume that the extension $K_{2}/K$ is the finite cyclic extension. Thus, we can write $B=\oplus B^{\chi^i}$ ([Ru], p.25). If $n$ denotes the rank of $B$ over $K_p/\mathscr{O}_p$, all components $B^{\chi^i}$ have the same rank $n$ over $K_p/\mathscr{O}_p$; otherwise the character $\chi$ factors through a smaller field. 
\end{proof}
\subsection{Main results}
\begin{thm}
For an elliptic curve $E/K$ with complex multiplication over an imaginary quadratic field $K$, we have
$$\ord_{s=1} L(E_{\slash K},s)=\rank(E(K)).$$
\end{thm}
\begin{proof}
As for the left vertical in the commutative diagram above, it follows from the snake lemma that the free rank of $E(\psi)^{G_K}$ is equal to that of the unit group 
$\mathscr{O}_{K_2}^{\times}$ (Lemma 4.3). Thus, we have $\ord_{s=1} L(E(\psi)_{\slash K},s)=\rank (E(\psi)^{G_{K}})$ ($\S3.3$) and this leads to $\ord_{s=1} L(E_{\slash K},s)=\rank(E(K))$ 
(Lemma 4.2).
\end{proof}
In particular, if $E$ is defined over $\mathbb{Q}$,   
since we have $L(E_{\slash K},s)=L(E_{\slash\mathbb{Q}},s)^2$ and $\rank(E(K))=2\cdot \rank(E(\mathbb{Q}))$ ([M], $\S$1), 
we obtain $\ord_{s=1} L(E_{\slash\mathbb{Q}},s)=\rank(E(\mathbb{Q}))$. Therefore, it follows that we can determine whether a given natural number $N$ is a congruent number by the criterion of Tunnell for the equation $E_N:y^2=x^3-N^2x$.

\end{document}